\documentclass{llncs}
\usepackage[utf8]{inputenc}

\usepackage[final]{graphicx}
\usepackage{graphicx}
\usepackage{amssymb,amsmath}
\usepackage{enumerate}

\def\N{\mathbb N}

\def\A{\mathcal A}
\def\B{\mathcal B}

\def\S{\mathcal S}

\def\uu{\mathbf u}

\newcommand{\GP}[1]{\mathrm{P}_{#1}} 
\newcommand{\UR}[1]{\mathrm{UR}_{#1}} 
\newcommand{\UL}[1]{\mathrm{UL}_{#1}} 

\usepackage{tikz}
\usepackage[]{algorithm2e}
\usepackage{nicefrac}
\usepackage[none]{hyphenat}

\title{An Algorithm Enumerating All Infinite Repetitions in a D0L-system}

\author{Karel Klouda\inst{1} \and Štěpán Starosta\inst{1}}

\institute{Faculty of Information Technology, Czech Technical University in Prague, Thákurova~9, 160~00, Prague, Czech Republic, \email{karel.klouda@fit.cvut.cz}}

\begin{document}

\maketitle

\begin{abstract}
	We describe a simple algorithm which, for a given D0L-system, returns all words $v$ such that $v^k$ is a factor of the language of the system for all $k$. This algorithm can be used to decide whether a D0L-system is repetitive.
\end{abstract}


\section{Introduction}

Repetitions in words were studied already by Thue~\cite{Th06} in 1906. Since then this topic has been addressed by many authors from many points of view. This work deals with repetitions in languages of  D0L-systems. In particular, we focus on the case when the language of a D0L-system contains an arbitrarily long repetition as a factor. The other case, when these repetitions are bounded, has been recently profoundly studied by Krieger~\cite{Kr07,Kr09}.

Ehrenfeucht and Rozenberg proved in~\cite{EhRo83} that an infinite repetition can appear in the language of a D0L-system as a factor only if there is a word $v$ such that $v^k$ is a factor of an element of the language for all $k \in \N$; formally, they proved that if a D0L-system is repetitive (i.e., there is an element of the language that contains a $k$-power as a factor for all positive $k$), then it is strongly repetitive (i..e, there exists a non-empty word $v$ such that $v^k$ appears in the language of the system as a factor for all positive $k$). Moreover, it follows from the work of Mignosi and Séébold in~\cite{MiSe} that this is the only source of infinite repetitions since for any D0L-system there is a constant $M$ such that $v^M$ being in the language as a factor implies that $v^k$ is in the language as a factor for all $k \in \N$.
Hence, in order to describe all infinite repetitions occurring in the language of a D0L-system as factors we can limit ourselves to all finite words $v$ having arbitrarily large repetitions in some element of the language.

The first problem we encounter is deciding whether the language of a given D0L-system contains factors with arbitrarily large repetitions or not. Decidability of this problem was proved for the first time in~\cite{EhRo83} and then, using different strategy, also in~\cite{MiSe}. But the algorithms are quite complicated and their complexity is unknown. Another algorithm working in polynomial time is given in~\cite{KoOt00} by Kobayashi and Otto.
Their approach uses the notion of \nohyphens{quasi-repetitive} elements: $v$ is a \emph{\nohyphens{quasi-repetitive} element} for a morphism $\varphi$ if there exist positive integers $n$ and $p$ with $p \neq 0$ such that $\varphi^n(v) = v_1^p$ where $v_1$ is a conjugate of $v$.
They proved that repetitiveness is equivalent to the existence of \nohyphens{quasi-repetitive} factors.

Here we slightly broaden this idea as we prove that a non-pushy D0L-system (pushy systems are known to be repetitive) is repetitive if and only if there is a letter $a$ and an integer $\ell$ such that the fixed point of $\varphi^\ell$ starting in $a$ is purely periodic.
For given $\ell$ and $a$, this is very easy to decide employing the algorithm introduced by Lando in~\cite{La91}.
Further, using the notion of simplification~\cite{EhRo78} and some technical results from~\cite{EhRo83} to handle pushy systems, we assemble an algorithm enumerating all arbitrarily long repetitions that appear in some element of the studied language.

The decidability of a more general problem whether $\varphi^\ell$ has an eventually periodic fixed point has been shown in \cite{Pa86} and \cite{HaLi86}.
Recently, Honkala in~\cite{Ho08} gave a simple algorithm to decide the problem.
However, the algorithm needs to compute a power of the morphism $\varphi$ that is greater than $k!$ where $k$ is the cardinality of the alphabet.

\section{Definitions and basic notions}

An \emph{alphabet} $\A$ is a finite set of \emph{letters}.
We denote by $\A^*$ the \emph{free monoid} on $\A$ and by $\A^+$ the set of all non-empty words.
The \emph{empty word} is denoted $\varepsilon$.
A subset of $\A^*$ is called a \emph{language} and its elements are \emph{words}.
Let $v = v_0 \cdots v_{n-1}$ with $v_i \in \A$ for $0 \leq i < n$.
The \emph{length} of $v$ is $n$ and is denoted by $|v|$.
We denote by $\mathrm{first}(v)$ the first letter of the word $v \in \A^+$, i.e., here $\mathrm{first}(v) = v_0$.
By repeating the word $v$ $k$-times with $k \in \N$ we get the \emph{$k$-power} of $v$ denoted by $v^k = vv\cdots v$.
Any infinite sequence of letters $\uu = u_0u_1\ldots$ is called an \emph{infinite word} over $\A$.
A word $v \in \A^*$ is a \emph{factor} of a finite or infinite word $u$ if there exist words $x$ and $y$ such that $u = xvy$; if $x$ is empty (resp. $y$ is empty), $v$ is a \emph{prefix} (resp. \emph{suffix}) of $u$. A word $w$ is a \emph{conjugate} of $v \in \A^*$ if $w = yx$ and $v = xy$ for some $x,y \in \A^*$; the set of all conjugates of $v$ is denoted by $[v]$. A word $v$ is \emph{primitive} if $v = z^k$ implies $k = 1$. The shortest word $x$ such that $v = x^k$, $k \in \N^+$, is the \emph{primitive root} of $v$. An infinite word $\uu$ is \emph{eventually periodic} if it is of the form $\uu = xyyyy\cdots = xy^\omega$; it is \emph{(purely) periodic} if $x$ is empty and \emph{aperiodic} if it is not eventually periodic.

Given two alphabets $\A$ and $\B$, any homomorphism $\varphi$ from $\A^*$ to $\B^*$ is called a \emph{morphism}. An infinite word $\uu$ is a \emph{periodic point} of a morphism $\varphi:\A^* \to \A^*$ if $\varphi^\ell(\uu) = \uu$ for some $\ell \geq 1$; periodic point starting with the letter $a$ such that $\varphi^\ell(a) = av$, for some $v \in \A^+$, is the infinite word $(\varphi^\ell)^\omega(a) = \lim_{k \to +\infty} \varphi^{k\ell}(a)$. A non-empty word $w$ is \emph{mortal} with respect to a morphism $\varphi$ if $\varphi^k(w) = \varepsilon$ for some $k$, otherwise it is \emph{immortal}.  A morphism $\varphi$ over $\A$ is \emph{non-erasing} if $\varphi(a)$ is non-empty for all $a \in \A$.

The triplet $G = (\A, \varphi, w)$, where $\A$ is an alphabet, $\varphi$ a morphism on $\A^*$ and $w \in \A^+$, is a \emph{D0L-system}.
The \emph{language} of $G$ is the set $L(G) = \{ \varphi^k(w) \mid k \in \N \}$.
The system $G$ is \emph{reduced} if every letter of $\A$ occurs in some element of $L(G)$.
In what follows, we will naturally suppose that we have a reduced system since if a system is not reduced, then we may consider a subset of the alphabet and a restriction of the morphism to get a reduced system with the same properties.
If the set $L(G)$ is finite, then $G$ is \emph{finite}. A letter $a \in \A$ is \emph{bounded} if the D0L-system $(\A, \varphi, a)$ is finite, otherwise $a$ is \emph{unbounded}.
The set of all bounded letters is denoted by $\A_0$.
The D0L-system $G$ is \emph{non-erasing} (resp. \emph{injective}) if the morphism $\varphi$ is non-erasing (resp. injective).
We denote by $S(L(G))$ the set of all factors of the elements of the set $L(G)$.

If for any $k \in \N^+$ there is a word $v$ such that $v^k \in S(L(G))$, then $G$ is \emph{repetitive}; if there is a word $v$ such that $v^k \in S(L(G))$ for all $k \in \N^+$, then $G$ is \emph{strongly repetitive}. By~\cite{EhRo83}, all repetitive D0L-systems are strongly repetitive. An important class of repetitive D0L-systems are pushy D0L-systems: $G$ is  \emph{pushy} if $S(L(G))$ contains infinitely many words over $\A_0$.

\section{Infinite periodic factors}

As explained above, to describe all infinite repetitions appearing as factors in the language of a D0L-system $G$ it suffices to study words $v$ such that $v^k \in S(L(G))$ for all $k \in \N$. Therefore we introduce the following notions.
\begin{definition}
	Given a D0L-system $G$, we say that $v^\omega$ is an \emph{infinite periodic factor} of $G$ if $v$ is a non-empty word and $v^k \in S(L(G))$ for all positive integers~$k$.

Let $v$ be non-empty. We say that infinite periodic factors $v^\omega$ and $u^\omega$ are \emph{equivalent} if the primitive root of $u$ is a conjugate of the primitive root of $v$.
We denote the equivalence class containing $v^\omega$ by~$[v]^\omega$.
\end{definition}

\subsection{Simplification}

In what follows, some of the proofs are based on the assumption that the respective D0L-system is injective. This does not mean any loss of generality of our results since any non-injective D0L-system can be mapped to an injective one that has the same structure of infinite periodic factors. This mapping is given by a simplification of morphisms~\cite{EhRo78}.
\begin{definition}
	Let $\A$ and $\B$ be two finite alphabets and let $f: \A^* \to \A^*$ and $g: \B^* \to \B^*$ be morphisms. We say that $f$ and $g$ are \emph{twined} if there exist morphisms $h: \A^* \to \B^*$ and $k: \B^* \to \A^*$ satisfying $k \circ h = f$ and $h \circ k = g$. If $\# \B < \# \A$ and $f$ and $g$ are twined, then $g$ is a \emph{simplification} (with respect to $(h,k)$) of $f$.
\end{definition}
If a D0L-system has no simplification, then it is called \emph{elementary}. It is known~\cite{EhRo78} that elementary D0L-systems are injective and thus non-erasing.

Moreover, it is easy to see that every non-injective morphism has a simplification that is injective.
A simple algorithm to find an injective simplification of a morphism works as follows.
If the morphism is erasing, one can find a non-erasing simplification using Proposition 3.6 in \cite{KoOt00}.
A simplification of a non-erasing morphism is related to the defect theorem and one can use the algorithm described in~\cite{HaKa}.

\begin{example}
	The morphism $f$ determined by $a \mapsto aca, b \mapsto badc, c \mapsto acab$ and $d \mapsto adc$ is not injective as $f(ab) = f(cd)$. Therefore, there must exist a simplification; let $h$ be a morphism given by $a \mapsto x, b \mapsto yz, c \mapsto xy, d \mapsto z$ and $k$ a morphism given by $x \mapsto aca, y \mapsto b, z \mapsto adc$. Since $f = k \circ h$, $f$ is twined with the morphism $g = h \circ k$, which is determined by $x \mapsto xxyx, y \mapsto yz, z \mapsto xzxy$ and defined over the alphabet $\{x,y,z\}$. One can easily check that $g$ has no simplification and hence it is elementary, and thus injective.
\end{example}

Let $G = (\A, f, w)$ and let $g: \B^* \to \B^*$ be an injective simplification of $f$ with respect to $(h,k)$.
We say that the D0L-system $(\B, g, h(w))$ is an \emph{injective simplification of $G$ (with respect to $(h,k)$}.
The role of injective simplifications in the study of repetitions is given by the following lemma.

\begin{lemma}[Kobayashi, Otto \cite{KoOt00}]\label{lem:simplifications_preserves_periodiv_factors}
	Let $G$ be a D0L-system and $G'$ its injective simplification with respect to $(h,k)$. Then $v^\omega$ is an infinite periodic factor of $G$ if and only if $(h(v))^\omega$ is an infinite periodic factor of $G'$.
\end{lemma}
With this lemma in hand, we can consider only injective D0L-systems in the sequel.

\subsection{Graph of infinite periodic factors}

If $v^\omega$ is an infinite periodic factor of $G = (\A, \varphi, w)$, then $(\varphi(v))^\omega$ is an infinite periodic factor, too.
This gives us a structure that can be captured as a graph.
\begin{definition}
	Let $G = (\A, \varphi, w)$ be a D0L-system.
	The \emph{graph of infinite periodic factors} of $G$, denoted $\GP{G}$, is a directed graph with loops allowed and defined as follows:
	\begin{enumerate}
		\item the set of vertices of $\GP{G}$ is the set
			$$
				V(\GP{G}) = \{ [v]^\omega \mid v^\omega \mbox{ is an infinite periodic factor of } S(L(G)) \};
			$$
		\item there is a directed edge from $[v]^\omega$ to $[z]^\omega$ if and only if $\varphi(v^\omega) \in [z]^\omega$.
	\end{enumerate}
\end{definition}
Obviously, the outdegree of any vertex of $\GP{G}$ is equal to one.
\begin{lemma} \label{lem:periodic_has_periodic_ancestor}
	If $G = (\A, \varphi, w)$ is an injective D0L-system, then any vertex $[v]^\omega \in \GP{G}$ has indegree at least $1$.
\end{lemma}

\begin{proof}
Let $[v]^\omega \in V(\GP{G})$.
Since $v^\omega$ is an infinite factor of $G$, there exists an infinite word $\uu$ such that $sv^\omega = \varphi(\uu)$ and the word $s$ is a non-empty suffix of $v$ such 	that $|s| < |v|$.

To obtain a contradiction suppose that $\uu$ is not eventually periodic, i.e., is aperiodic.
By the pigeonhole principle there exist words $v_1$ and $v_2$ such that $v = v_1v_2$ and there exist infinitely many prefixes $u$ of $\uu$ such that
$\varphi(u) = s v^kv_1$ for some integer $k$.

Let $(u_1, u_2, u_3)$ be a subsequence  of the sequence of all such prefixes ordered by increasing length.
Denote $u_1 x = u_2$ and $u_2 y = u_3$. The situation is demonstrated in Figure \ref{fig:dukaz}.
It is known (see, e.g., Proposition 1.3.2 in \cite{Lo83}) that for any $w_1,w_2 \in \A^*$ the equality $w_1w_2 = w_2w_1$ implies that there exists a word $z$ such that $w_1 = z^m$ and $w_2 = z^n$ for some integers $m$ and $n$. 
Therefore, since $\uu$ is aperiodic, we can choose the prefixes $u_1, u_2$ and $u_3$ such that $xy \neq yx$. We have $\varphi(x) = v_2v^kv_1$ and $\varphi(y) = v_2v^\ell v_1$ for some integers $k$ and $\ell$.
This implies $\varphi(xy) = \varphi(yx)$ which is a contradiction.

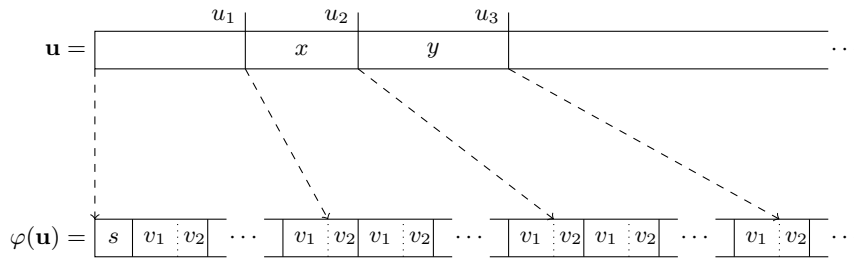
\begin{figure}
\centering
\begin{tikzpicture}
	\node [left] at (1.5,4) {$\mathbf{u} = $};
	\node at (11.5,4) {$\cdots$};
	\draw (11.25,4.25) -- (1.5,4.25) -- (1.5, 3.75) -- (11.25,3.75);
	\draw (3.5,3.75) -- (3.5,4.5);
	\node [above left] at (3.5,4.25) {$u_1$};
	\draw (5,3.75) -- (5,4.5);
	\node [above left] at (5,4.25) {$u_2$};
	\draw (7,3.75) -- (7,4.5);
	\node [above left] at (7,4.25) {$u_3$};
	\node at (4.25,4) {$x$};
	\node at (6,4) {$y$};
	\node [left] at (1.5,1.5) {$\varphi(\mathbf{u}) = $};
	\draw (3.25,1.75) -- (1.5,1.75) -- (1.5, 1.25) -- (3.25,1.25);
	\draw (2,1.25) -- (2,1.75);
	\node at (1.75,1.5) {$s$};
	\node at (2.3,1.5) {$v_1$};
	\draw [dotted] (2.6,1.25) -- (2.6,1.75);
	\node at (2.83,1.5) {$v_2$};
	\draw (3,1.25) -- (3,1.75);
	\node at (3.5,1.5) {$\cdots$};

	\draw (6.25,1.75) -- (3.75,1.75);
	\draw (3.75, 1.25) -- (6.25,1.25);
	\draw (4,1.25) -- (4,1.75);
	\node at (4.3,1.5) {$v_1$};
	\draw [dotted] (4.6,1.25) -- (4.6,1.75);
	\node at (4.83,1.5) {$v_2$};
	\draw (5,1.25) -- (5,1.75);
	\node at (5.3,1.5) {$v_1$};
	\draw [dotted] (5.6,1.25) -- (5.6,1.75);
	\node at (5.83,1.5) {$v_2$};
	\draw (6,1.25) -- (6,1.75);
	\node at (6.5,1.5) {$\cdots$};

	\draw (9.25,1.75) -- (6.75,1.75);
	\draw (6.75, 1.25) -- (9.25,1.25);
	\draw (7,1.25) -- (7,1.75);
	\node at (7.3,1.5) {$v_1$};
	\draw [dotted] (7.6,1.25) -- (7.6,1.75);
	\node at (7.83,1.5) {$v_2$};
	\draw (8,1.25) -- (8,1.75);
	\node at (8.3,1.5) {$v_1$};
	\draw [dotted] (8.6,1.25) -- (8.6,1.75);
	\node at (8.83,1.5) {$v_2$};
	\draw (9,1.25) -- (9,1.75);
	\node at (9.5,1.5) {$\cdots$};

	\draw (11.25,1.75) -- (9.75,1.75);
	\draw (9.75, 1.25) -- (11.25,1.25);
	\draw (10,1.25) -- (10,1.75);
	\node at (10.3,1.5) {$v_1$};
	\draw [dotted] (10.6,1.25) -- (10.6,1.75);
	\node at (10.83,1.5) {$v_2$};
	\draw (11,1.25) -- (11,1.75);
	\node at (11.5,1.5) {$\cdots$};

	\draw[dashed,->] (1.5,3.75) -- (1.5,1.75);
	\draw[dashed,->] (3.5,3.75) -- (4.6,1.75);
	\draw[dashed,->] (5,3.75) -- (7.6,1.75);
	\draw[dashed,->] (7,3.75) -- (10.6,1.75);
\end{tikzpicture}
\caption{The choice of prefixes $u_1$, $u_2$ and $u_3$.}
 \label{fig:dukaz}
\end{figure}

Thus, $\uu$ is an infinite periodic factor of $G$.
It follows that there exists a primitive word $u$ such that $\uu = zu^\omega$, $z \in \A^*$, and there is a directed edge in $\GP{G}$ from $[u]^\omega$ to $[v]^\omega$. \qed

\end{proof}

\begin{corollary}
If $G = (\A, \varphi, w)$ is an injective D0L-system,
then its graph of infinite periodic factors $\GP{G}$ is $1$-regular. In other words, $P_G$ consists of disjoint cycles.
\end{corollary}
So, each vertex of $\GP{G}$ is a vertex of a cycle. In what follows, we will distinguish two types of cycles:
\begin{lemma}
	Let $G = (\A, \varphi, w)$ be an injective D0L-system. If $[v]^\omega$ and $[w]^\omega$ are two vertices of the same cycle in $\GP{G}$, then $v$ consists of bounded letters if and only if $w$ consists of bounded letters.
\end{lemma}
A proof follows from the obvious fact that if $v \in \A_0^+$, then $\varphi^k(v) \in \A_0^+$ for all $k \in \N^+$. Similarly, if $v$ contains an unbounded letter, then $\varphi^k(v)$ must contain an unbounded letter as well for all $k \in \N^+$. So, in $\GP{G}$ we have two types of cycles: over bounded letters only or containing some unbounded letters. We treat these two cases separately.

\subsection{Bounded letters}

Obviously, if $v^\omega$ is an infinite periodic factor of a D0L-system $G$ with $v$ containing bounded letters only, then $G$ is pushy by definition.
In \cite{EhRo83} it is proved that it is decidable whether a D0L-system is pushy or not. In particular, the authors proved that a D0L-system is pushy if and only if it satisfies the \emph{edge condition}: there exist $a \in \A$, $k \in \N^+$, $v \in \A^*$ and $u \in \A_0^+$ such that $\varphi^k(a) = vau$ or $\varphi^k(a) = uav$. Here we slightly broaden the idea and show how to describe all factors over $\A_0$.
\begin{definition}
Let $G = (\A, \varphi, w)$ be a D0L-system.
\emph{The graph of unbounded letters of $G$ to the right}, denoted $\UR{G}$, is the labeled directed graph defined as follows:
	\begin{enumerate}[(i)]
		\item the set of vertices is $V(\UR{G}) = \A \setminus \A_0$,
		\item there is a directed edge from $a$ to $b$ with label $u \in \A_0^*$ if there exists $v \in \A^*$ such that $\varphi(a) = vbu$.
	\end{enumerate}
	\emph{The graph of unbounded letters of $G$ to the left} $\UL{G}$ is defined analogously: the only difference is that the roles of $v$ and $u$ in the definition of a directed edge are switched.
\end{definition}
Clearly, the edge condition is satisfied if and only if one of the graphs $\UL{G}$ and $\UR{G}$ contains a cycle including an edge with an immortal label. Therefore, we have this:
\begin{proposition}
	A D0L-system $G$ is pushy if and only if one of the graphs $\UL{G}$ and $\UR{G}$ contains a cycle including an edge with an immortal label.
\end{proposition}

Since the outdegree of any vertex of the graphs $\UL{G}$ and $\UR{G}$ is $1$, the graphs consist of components each containing exactly one cycle.
Assume that $a$ is a vertex of a cycle of $\UR{G}$ that contains at least one edge with a non-empty immortal label, i.e., there exist $k \in \N^+$, $u_1, \ldots, u_k \in \A_0^*$ and $a_1, \ldots, a_{k-1} \in \A \setminus \A_0$ so that
$$
	a \xrightarrow{u_1} a_1 \xrightarrow{u_2} \cdots \xrightarrow{u_{k-1}} a_{k-1} \xrightarrow{u_{k}} a
$$
is a cycle of $\UR{G}$ with $u_j$ being immortal for some $j$ such that $1 \leq j \leq k$.
Denote $u = u_{k} \varphi(u_{k-1}) \cdots \varphi^{k-1}(u_1)$, then for all $\ell \in \N^+$ the word $\varphi^{\ell k}(a)$ has a suffix
$$
	u\varphi^{k}(u)\varphi^{2k}(u)\cdots\varphi^{(\ell - 1)k}(u) \in \A_0^+\,.
$$

\begin{lemma}
	Let $G = (\A, \varphi, w)$ be a D0L-system. If $u \in \A_0^+$ is immortal and $k \in \N^+$, then the infinite word
	$$
		\mathbf{u} = u\varphi^{k}(u)\varphi^{2k}(u)\varphi^{3k}(u) \ldots
	$$
	is eventually periodic.
\end{lemma}
\begin{proof}
	Clearly, the sequence $\left ( \varphi^j(u) \right)_{j=0}^{+\infty} $ is eventually periodic. Let $s$ and $t$ be numbers such that $\varphi^{s}(u) = \varphi^{s+t}(u)$ with $t > 0$. Define $\ell_0$ and $\ell_1$ so that
	$\ell_0k \geq s$ and $(\ell_1 - \ell_0)k$ is a multiple of $t$. We get that
	$$
		\mathbf{u} = u\varphi^k(u)\cdots\varphi^{\ell_0 k}(u)\left(\varphi^{(\ell_0 + 1)k}(u)\varphi^{(\ell_0 + 2)k}(u)\cdots \varphi^{\ell_1 k}(u)\right)^\omega.
	$$
	\qed 
\end{proof}
The situation is analogous for the graph $\UL{G}$.
Since the cycles of the graphs $\UR{G}$ and $\UL{G}$ are clearly the only sources of factors over $\A_0$ of arbitrary length, we get the following theorem.
This result has been already obtained in \cite{CANT_2010_CaNi}, Proposition 4.7.62, where one can find a distinct proof.
\begin{theorem}\label{thm:all_pushy_factors}
	If $G$ is a pushy D0L-system, then there exist $L \in \N$ and a finite set $\mathcal{U}$ of words from $\A_0^+$ such that any factor from $\S(L(G)) \cap \A_0^+$ is of one of the following three forms:
	\begin{enumerate}[(i)]
		\item $w_1$,
		\item $w_1u_1^{k_1}w_2$,
		\item $w_1u_1^{k_1}w_2u_2^{k_2}w_3$,
	\end{enumerate}
	where $u_1, u_2 \in \mathcal{U}$, $|w_j| < L$ for all $j \in \{1,2,3\}$, and $k_1, k_2 \in \N^+$.
\end{theorem}

\begin{example}
	Consider the D0L-system $G = (\A, \varphi, 0)$ with $\A = \{0,1,2\}$ and $\varphi$ determined by $0 \mapsto 012, 1\mapsto 2$ and $2 \mapsto 1$. There are two bounded letters, $\A_0 = \{1,2\}$, and one unbounded letter. The graphs $\UL{G}$ and $\UR{G}$ both contain one loop on the only vertex $0$ labeled with the empty word and $12$, respectively. Therefore, $G$ is pushy and the corresponding infinite periodic factor over bounded letters is a suffix of
	$$
		12\varphi(12)\varphi^2(12)\varphi^3(12)\ldots\,,
	$$
	namely $(1221)^\omega$. In fact, this infinite periodic factor is a suffix of the eventually periodic fixed point of $\varphi$ starting in $0$.

	The D0L-system $H = (\B, \psi, 0)$ with $\B = \{0,1,2,3\}$ and $\psi$ determined by $0 \mapsto 0123, 1\mapsto 2, 2 \mapsto 1$ and $3 \mapsto 123$ is also pushy. The graph $\UL{H}$ contains two loops: one on the vertex $0$ labeled with the empty word and one on the vertex $3$ labeled with $12$; the labels in $\UR{H}$ are all equal to the empty word. Hence, there is an infinite periodic factor over bounded letters which is an infinite prefix of the left-infinite word
	$$
		\ldots\varphi^3(12)\varphi^2(12)\varphi(12)12.
	$$
	It this case, it is the word itself, namely $(2112)^\omega$. In this case, the fixed point of $\psi$ is not eventually periodic, but still all prefixes of $(2112)^\omega$ are its factors.
\end{example}

\subsection{Unbounded letters}

Now we address the other case: infinite periodic factors containing an unbounded letter.
\begin{theorem}\label{thm:periodic_periodic_points}
	Let $G = (\A, \varphi, w)$ be an injective repetitive D0L-system. If $[v]^\omega$ is a vertex of $\GP{G}$ such that $v$ contains an unbounded letter, then there exist $b \in \A$ and $\ell \in \N, 1 \leq \ell \leq \#\A$, such that
	$(\varphi^\ell)^\omega(b) = z^\omega$ for some $z \in [v]$.
\end{theorem}

\begin{proof}
	Clearly, we can assume that $v$ is primitive. If $[v]^\omega$ is a vertex of a cycle of length $m$ in the graph $\GP{G}$, then $\varphi^m(v^\omega) \in [v]^\omega$. Since $v$ contains an unbounded letter, we get $|\varphi^m(v)| > |v|$ and so there exist a nonnegative integer $k_1$, a suffix $s_1$ of $v$ and a prefix $p_1$ of $v$ such that $\varphi^m(v) = s_1v^{k_1}p_1$. In fact, $\varphi^{tm}(v)$ is a factor of $v^\omega$ for all $t$, thus there exist a sequence of integers $(k_t)_{t \geq 1}$, sequence of suffixes of $v$ $(s_t)_{t \geq 1}$ and prefixes of $v$ $(p_t)_{t \geq 1}$ such that $\varphi^{tm}(v) = s_tv^{k_t}p_t$. Clearly it must hold that $p_ts_t = v$ for all $t$ (since $\varphi^{tm}(vv)$ is also a factor of $v^\omega$).

	Obviously, the sequence $(s_t)_{t \geq 1}$ is eventually periodic (and so is $(p_t)_{t \geq 1}$). Therefore, there exist integers $t_1 < t_2$ such that $s_{t_1} = s_{t_2}$. It follows that $\varphi^{m(t_2 - t_1)}(s_{t_1}p_{t_1}) = (s_{t_1}p_{t_1})^k$ for some $k > 1$.

	Put $z = s_{t_1}p_{t_1}$, $\ell = m(t_2 - t_1)$ and $b = \mathrm{first}(z)$. Clearly, $\varphi^{k\ell}(b)$ is a prefix of $z^\omega$ for all $k$. If $b$ is unbounded, we are done. Assume $b$ is bounded, then there exists $f$ such that $u = \varphi^f(b) = \varphi^{2f}(b)$. Hence, $u$ is a prefix of $z = uz'$ with $|z'| > 0$ and we get $\varphi^{\ell}(z'u) = (z'u)^k$ for some $k > 1$. Since $z$ must contain at least one unbounded letter and $u$ consists of bounded letters only, we can repeat this until the first letter of $z'$ is unbounded.

	Assume now that $(\varphi^{\ell})^\omega(b) = z^\omega$. It remains to prove that $\ell$ can be taken less than or equal to $\#\A$.
	Put $\ell_{\min} = \min\{j \in \N \colon \mathrm{first}(\varphi^j(b)) = b \}$.
	Obviously such $\ell_{\min}$ exists and is at most $\#\A$. Moreover, $\ell$ must be a multiple of $\ell_{\min}$ and hence we can put $\ell = \ell_{\min}$ and still have $(\varphi^{\ell})^\omega(b) = z^\omega$. \qed
\end{proof}

By Theorems~\ref{thm:all_pushy_factors} and \ref{thm:periodic_periodic_points} and by Lemma~\ref{lem:simplifications_preserves_periodiv_factors} we get the following claim.
\begin{corollary}
	Any repetitive D0L-system contains a finite number of primitive words $v$ such that $v^\omega$ is an infinite periodic factor.
\end{corollary}

\section{Algorithm}

\begin{algorithm}[h!]

\SetKwInOut{Input}{Input}\SetKwInOut{Output}{Output}
\SetKwData{ol}{OutputList}
\SetKwData{wu}{u}
\SetKwData{cy}{c}
\SetKwData{es}{s}
\SetKwData{te}{t}
\SetKwData{lo}{l0}
\SetKwData{lp}{l1}
\SetKwData{la}{a}
\SetKwData{el}{lh}
\SetKwData{elll}{l}
\SetKwData{ve}{v}
\SetKwData{lh}{h}
\SetKwData{lk}{k}
\SetKwFunction{lcm}{LeastCommonMultiple}
\SetKwFunction{pw}{PrimitiveRoot}
\SetKwFunction{fl}{FirstLetter}
\SetKwFunction{break}{break}
\SetKwFunction{return}{return}

\Input{Alphabet $\A$, morphism $\varphi: \A^* \to \A^*$.}
\Output{List of all primitive words $v$ such that $v^k$ can be generated by $\varphi$ for all $k$.}
\BlankLine

$\varphi \leftarrow$ injective simplification of $\varphi$ with respect to (\lh,\lk)\;
calculate the list of bounded letters $\A_0$\;
construct the graphs $\UL{G}$ and $\UR{G}$\;
\ol $\leftarrow$ empty list\;
\ForEach{cycle \cy in $\UL{G}$}{
	\If{\cy contains an edge with non-empty immortal label}
	{
		\wu $\leftarrow  u_{k} \varphi(u_{k-1}) \cdots \varphi^{k-1}(u_1)$ where $u_1,\ldots,u_k$ are the labels of edges in the cycle \cy\;
		find the least $\es$ and $\te$ such that $\te > 0$ and $\varphi^{\es}(\wu) = \varphi^{\es+\te}(\wu)$\;
		$\lo \leftarrow \lceil \nicefrac{\es}{k} \rceil \cdot k$\;
		$\lp \leftarrow \lo + \lcm(\te, k) / k$\;
		append \pw($\varphi^{(\lo + 1)k}(\wu)\varphi^{(\lo + 2)k}(\wu)\cdots \varphi^{\lp k}(\wu)$) to \ol\;
	}
}
\ForEach{cycle \cy in $\UR{G}$}{
	//analogous procedure as for cycles in $\UL{G}$
}
\ForEach{letter \la in $\A \setminus \A_0$}{
	find the least \elll such that $\fl(\varphi^{\elll}(\la)) = \la$ with $\elll \leq \# \A$ \;
	\If{$\elll$ exists}{
		find the least $\es$ such that $\varphi^{\elll \cdot \es}(\la)$ contains at least two occurrences of one unbounded letter \;
		\If{$\varphi^{\elll \cdot \es}(\la)$ contains at least two occurrences of $\la$}
		{
			$\ve \leftarrow$ the longest prefix of $\varphi^{\elll \cdot \es}(\la)$ containing only one occurrence of $\la$\;
			\If{$\varphi^{\elll}(\ve) = \ve^{m}$ for some integer $m \geq 2$}
			{
				append \ve to \ol\;
			}
		}
	}
}
\ol $\leftarrow \{ \pw(\lk(w)) \colon w \in \ol\}$\;
add conjugates to \ol\;
\return \ol\;
\BlankLine
\caption{Pseudocode for the main algorithm.}\label{algo}
\end{algorithm}

Given an injective D0L-system $G = (\A, \varphi, w)$, all infinite periodic factors over bounded letters can be obtained from the cycles in graphs $\UL{G}$ and $\UR{G}$.
It remains unclear how to find infinite periodic factors containing an unbounded letter.
Theorem~\ref{thm:periodic_periodic_points} says that equivalence classes $[v]^\omega$ that contains those factors are in one-to-one correspondence with periodic periodic\footnote{Note that ``periodic periodic'' is not a typing error: it is a periodic point of a morphism, which happens to be periodic, i.e., of the form $w^\omega$.} points of the morphism $\varphi$.
Consider the following \emph{graph of first letters} over the set of vertices  equal to $\A$: there is a directed edge from $a$ to $b$ if $b = \mathrm{first}(\varphi(a))$. It follows that $(\varphi^\ell)^\omega(a)$ is an infinite periodic point of $\varphi$ if and only if $a \in \A\setminus
\A_0$ is a vertex of a cycle of a length that divides $\ell$. Therefore, we have only finitely many candidates $a \in \A$ and $\ell \in \N^+$ for which we need to verify whether $(\varphi^\ell)^\omega(a)$ is a periodic infinite word. In fact, it suffices to check only one letter from each cycle of the graph of first letters. Finally, to verify whether the word $(\varphi^\ell)^\omega(a)$ is periodic we can use the algorithm described in~\cite{La91}. This algorithm is very effective, in short it works as follows:
\begin{enumerate}
	\item Find the least $k \leq \#\A$ such that $(\varphi^\ell)^k(a)$ contains at least two occurrences of one unbounded letter.
	\item If $(\varphi^\ell)^k(a)$ does not contain two occurrences of $a$, then $(\varphi^\ell)^\omega(a)$ is not periodic. Otherwise denote by $v$ the longest prefix of $(\varphi^\ell)^k(a)$ containing $a$ only as the first letter.
	\item Now, $(\varphi^\ell)^\omega(a)$ is periodic if and only if $\varphi^\ell(v) = v^m$ for some integer $m \geq 2$.
\end{enumerate}
This, together with the algorithm to construct an injective simplification, gives us an effective algorithm that decides whether a given D0L-system is repetitive and, moreover, returns the list of all infinite periodic factors.

The described algorithm is given in pseudocode in Algorithm \ref{algo}.

\section*{Acknowledgement}

We acknowledge financial support by the Czech Science Foundation grant GA\v CR 13-35273P.
The algorithm was implemented and tested in the open-source mathematical software \texttt{Sage} \cite{sage5}.

\bibliographystyle{siam}
\bibliography{biblio}

\end{document}